\documentclass[11pt]{article}

\usepackage{amsmath,amssymb,amsthm}
\usepackage{bm,mathrsfs}
\usepackage{microtype}
\usepackage{graphicx}

\newcommand{\abs}[1]{\left|#1\right|}
\newcommand{\norm}[1]{\left\|#1\right\|}

\newtheorem{theorem}{Theorem}[section]

\newtheorem{proposition}[theorem]{Proposition}

\theoremstyle{definition}
\newtheorem{definition}[theorem]{Definition}
\newtheorem{property}[theorem]{Property}

\title{Harmonic Singular Integrals and Steerable Wavelets in $L_2(\mathbb{R}^d)$  \thanks{This research was funded in part by ERC Grant ERC-2010-AdG 267439-FUN-SP.}}

\author{
John Paul Ward \thanks{Biomedical Imaging Group, \'Ecole polytechnique f\'ed\'erale de Lausanne (EPFL),
Station 17, CH-1015, Lausanne, Switzerland  ({\tt john.ward@epfl.ch}).}
\and
Michael Unser \thanks{Biomedical Imaging Group, \'Ecole polytechnique f\'ed\'erale de Lausanne (EPFL),
Station 17, CH-1015, Lausanne, Switzerland  ({\tt michael.unser@epfl.ch}).}
 }

\begin{document}

\maketitle

\begin{abstract}
Here we present a method of constructing steerable wavelet frames in $L_2(\mathbb{R}^d)$ that generalizes and unifies previous approaches, including Simoncelli's pyramid and Riesz wavelets.
The motivation for steerable wavelets is the need to more accurately account for the orientation of data.  Such wavelets can be constructed by decomposing an isotropic mother wavelet into a finite collection of oriented mother wavelets. The key to this construction is that the angular decomposition is an isometry, whereby the new collection of wavelets maintains the frame bounds of the original one.  The general method that we propose here is based on partitions of unity involving spherical harmonics.  A fundamental aspect of this construction is that Fourier multipliers composed of spherical harmonics correspond to singular integrals in the spatial domain.  Such transforms have been studied extensively in the field of harmonic analysis, and we take advantage of this wealth of knowledge to make the proposed construction practically feasible and computationally efficient.
\end{abstract}

\section{Introduction}\label{sec:intro}
Building upon \cite{unser13}, our purpose in this paper is to provide a systematic and practical approach to the construction and implementation of steerable wavelets in higher dimensions.  The basis of this construction is the theory of singular integral transforms on $\mathbb{R}^d$ with kernels of the form $\Omega(\bm{x}/\abs{\bm{x}})/\abs{\bm{x}}^d$, where $\Omega$ is a smooth function defined on $\mathbb{S}^{d-1}$.  Properties of such transforms were studied by Mikhlin \cite{mikhlin65}, and Calder\'on and Zygmund \cite{calderon57}. Useful resources for this material are the books of Stein and Weiss, \cite{stein70,stein71}, which we shall use as primary references. An attractive feature of singular integral transforms is their correspondence with Fourier multiplier transforms. For instance, if $\Omega$ is a spherical harmonic, then the singular integral transform corresponds to a Fourier multiplier that is a multiple of $\Omega$.  Furthermore, the spherical harmonics (in particular, the zonal spherical harmonics) satisfy symmetry properties which make them ideal for applications requiring rotations.

The two key ingredients for our construction of steerable wavelets are:
\begin{itemize}
\item[1)] an isotropic, band-limited mother wavelet $\psi$ that generates a primary wavelet frame of $L_2(\mathbb{R}^d)$;
\item[2)] a finite collection of functions $\{m_n\}_{n=1}^{n_{\text{max}}}$ which generate a partition of unity on the sphere:
\begin{equation*}
\sum_{n=1}^{n_{\text{max}}} \abs{m_n(\bm{\omega})}^2=1,
\end{equation*}
where the $m_n$ are purely polar functions; i.e., $m_n(\bm{\omega})=m_n(\bm{\omega}/\abs{\bm{\omega}})$. 
\end{itemize}

\noindent
The steerable wavelet frame is then generated by the functions $\{\mathcal{F}^{-1}\{m_n\widehat{\psi}\} \}$.  The limitation on the functions $m_n$ are minimal; however, we shall focus on zonal spherical harmonics as they make implementation more amenable.  Notice that  decomposing a signal in this enlarged dictionary of wavelets provides more information about the local orientation of the data; meanwhile, the partition of unity property guarantees that the frame bounds are preserved.

We shall devote the remainder of this section to some basic notation. In Section \ref{sec:sing_int}, we shall recall some results about singular integral transforms and their relation to tight frames.  In Section \ref{sec:sph_harm}, we shall cover some details about spherical harmonics and properties of related transforms.  Finally, in Section \ref{sec:steer_wave}, we shall describe the steerable wavelet construction, and in Section \ref{sec:zonal_harm} we conclude with some specifics about wavelets based on spherical harmonics.

\subsection{Notation}
The function spaces that we shall consider are the $L_p(\mathbb{R}^d)$ spaces, with norm

\begin{equation*}
\norm{f}_{L_p(\mathbb{R}^d)} = \left( \int_{\mathbb{R}^d} \abs{f(\bm{x})}^p {\rm d}\bm{x} \right)^{1/p},
\end{equation*} 
for $1\leq p <\infty$, and our primary focus shall be with $L_2(\mathbb{R}^d)$.  Following the notation of \cite{stein71}, we define the Fourier transform of a function $f\in L_1(\mathbb{R}^d)$ to be

\begin{align*}
\widehat{f}(\bm{\omega})=\mathcal{F}\{f\}(\bm{\omega})&=\int_{\mathbb{R}^d}f(\bm{x})\overline{e^{2\pi i \bm{x}\cdot \bm{\omega}}} {\rm d}\bm{x}\\
&=\int_{\mathbb{R}^d}f(\bm{x})e^{-2\pi i \bm{x}\cdot \bm{\omega}} {\rm d}\bm{x}
\end{align*}
and the inverse Fourier transform of $f$ is denoted by $\mathcal{F}^{-1}\{f\}$.  For a radial function $f(\bm{x})=f_r(\abs{\bm{x}})$, we can write the Fourier transform as
 
\begin{align*}
\mathcal{F}\{f\}(\bm{\omega})&=\int_0^{\infty} t^{d-1} f_r(t) \int_{\mathbb{S}^{d-1}} e^{-2\pi i t(\bm{x}/\abs{\bm{x}})\cdot \bm{\omega}} {\rm d}\sigma \left(\frac{\bm{x}}{\abs{\bm{x}}}\right) {\rm d}t\\
&= 2\pi \abs{\bm{\omega}}^{-(d-2)/2} \int_0^\infty f_r(t) J_{(d-2)/2}(2\pi \abs{\bm{\omega}t})t^{d/2} {\rm d} t
\end{align*}
where $\sigma$ is the usual surface measure on $\mathbb{S}^{d-1}$ and $J_{(d-2)/2}$ is the Bessel function of the first kind of order $(d-2)/2$, cf. 
\cite[Section VIII.3]{stein93}. The area of the sphere is given by

\begin{equation*}
\sigma(\mathbb{S}^{d-1}) = \frac{2\pi^{d/2}}{\Gamma(d/2)},
\end{equation*}
where $\Gamma$ denotes the Gamma function.

For vector valued functions $f:\mathbb{R}^d\rightarrow \mathbb{C}^N$, we denote the $n$th component by $[f]_n$. The space of such functions, all of whose components are $L_2(\mathbb{R}^d)$ functions, will be denoted by $L_2^N(\mathbb{R}^d)$, and we define the norm
\begin{equation*}
\norm{f}_{L_2^N(\mathbb{R}^d)}=\left(\sum_{j=1}^N \norm{[f]_n}_{L_2(\mathbb{R}^d)}^2 \right)^{1/2}.
\end{equation*}

\section{Singular integrals and Fourier multipliers}\label{sec:sing_int}

In this section, we recall some relevant results from the theory of singular integrals and provide a basis for the construction of steerable wavelets.  One of the key ingredients in this construction is a collection of functions which generate a partition of unity. In this section we shall show how particular classes of such functions behave as Fourier multipliers.

\begin{definition}\label{def:adm}
A collection of complex valued functions $\mathcal{M}=\{m_n\}_{n=1}^{n_{\text{max}}}$ will be called admissible if 
\begin{enumerate}
\item Each $m_n$ is Lebesgue measurable and homogeneous of degree $0$; i.e., $m_n(a \bm{\omega})=a^0 m_n(\bm{\omega})=m_n(\bm{\omega})$ for all $a>0$ and $\bm{\omega}\neq 0$;
\item The squared moduli of the elements of $\mathcal{M}$ form a partition of unity:
\begin{equation*}
\sum_{n=1}^{n_{\text{max}}} \abs{m_n(\bm{\omega})}^2=1
\end{equation*}
for every $\bm{\omega}\in \mathbb{R}^{d}\backslash \{\bm{0}\}$.
\end{enumerate}
\end{definition}

\noindent
The partition of unity property implies that $\abs{m_n(\bm{\omega})}^2\leq 1$, so each function is a valid Fourier multiplier on $L_2(\mathbb{R}^d)$; i.e.
\begin{equation*}
\norm{\mathcal{F}^{-1}\{m_n\widehat{f}\} }_{L_2(\mathbb{R}^d)} \leq \norm{f}_{L_2(\mathbb{R}^d)}.
\end{equation*}
Therefore, we can define a transform mapping $L_2(\mathbb{R}^d)$ to the vector valued space $L_2^{n_{\text{max}}}(\mathbb{R}^d)$ as follows.

\begin{definition}
Given an admissible collection $\mathcal{M}$, define the transform $T_{\mathcal{M}}:L_2(\mathbb{R}^d)\rightarrow L_2^{n_{\text{max}}}(\mathbb{R}^d) $ by
\begin{equation*}
\left[T_{\mathcal{M}}(f) \right]_n=\mathcal{F}^{-1}\{m_n \widehat{f}\},
\end{equation*}
and its adjoint $T_{\mathcal{M}}^*:L_2^{n_{\text{max}}}(\mathbb{R}^d) \rightarrow L_2(\mathbb{R}^d) $ by
\begin{equation*}
T_{\mathcal{M}}^*(f) = \mathcal{F}^{-1} \left\{ \sum_n \overline{m}_n \widehat{[f]_n}  \right\}.
\end{equation*}

\end{definition}

\noindent
Note that, since we apply these transforms to wavelets, the homogeneity condition makes sense, as it means that $m_n$ is invariant to scaling. 

\begin{property}
The transform $T_{\mathcal{M}}$ maps a wavelet family into another one in the sense that 
\begin{equation*}
\left[T_{\mathcal{M}}(\psi(\cdot/a-\bm{b})) \right]_n (\bm{x}) = \left[T_{\mathcal{M}}(\psi) \right]_n (\bm{x}/a-\bm{b})
\end{equation*}
for any $\psi \in L_2(\mathbb{R}^d)$, $a\in \mathbb{R}^+$, and $\bm{b}\in\mathbb{R}^d$.
\end{property}

 Also, it follows from Plancherel's identity and the partition of unity condition that $T_{\mathcal{M}}$ is in fact an isometry. Hence, we can apply this transform to a tight frame to generate a new frame with the same frame bounds.

\begin{theorem}\label{th:ghrtpf}
Suppose $\{\phi_k:k\in\mathbb{Z}\}$ is a Parseval frame for $L_2(\mathbb{R}^d)$; i.e.,
\begin{equation}\label{eq:frame_expan}
f = \sum_{k} \left< f, \phi_k  \right> \phi_k,
\end{equation}
and

\begin{equation}\label{eq:frame_bds}
 \sum_k \abs{\left< f, \phi_k \right>}^2 = \norm{f}_2^2
\end{equation}
for every $f\in L_2(\mathbb{R}^d)$.  Then the multiplier transform $T_{\mathcal{M}}$ associated with an admissible collection $\mathcal{M}$ generates a related Parseval frame:  
\begin{equation*}
\left\{\psi_{k,n}=[T_{\mathcal{M}}(\phi_k)]_n:k\in\mathbb{Z},n=1,\dots ,n_{\text{max}} \right\},
\end{equation*}
with 
\begin{align*}
f &= \sum_{n=1}^{n_{\text{max}}} \sum_k \left< f, \psi_{k,n}  \right> \psi_{k,n}.
\end{align*}
\end{theorem}

\begin{proof}
Our proof follows the same lines as \cite[Proposition 1]{unser10}. Notice that if $\mathcal{M}$ is admissible, then so is $\overline{\mathcal{M}}:=\{\overline{m}_n\}_{n=1}^{n_{\text{max}}}$. From the definition of $\mathcal{M}$
\begin{equation*}
f = T_{\mathcal{\overline{M}}}^* T_{\mathcal{\overline{M}}} f.
\end{equation*}
Additionally, each component of $ T_{\mathcal{\overline{M}}} f  $ can be expanded in the original frame
\begin{equation*}
f = T_{\mathcal{\overline{M}}}^* \mathbf{F}
\end{equation*}
where $\mathbf{F}$ is the function with components
\begin{align*}
[\mathbf{F}]_n &= \sum_{k} \left< [T_{\mathcal{\overline{M}}}(f)]_n, \phi_k  \right> \phi_k\\
&= \sum_{k} \left< f, [T_{\mathcal{M}}(\phi_k)]_n \right> \phi_k\\
\end{align*}
The reproduction property now follows by computing the product $ T_{\mathcal{\overline{M}}}^* \mathbf{F}$:
\begin{align*}
f &= \mathcal{F}^{-1} \left\{\sum_n m_n \sum_k \left< f, [T_{\mathcal{M}}(\phi_k)]_n  \right> \widehat{\phi}_k \right\}\\
&= \sum_n \sum_k \left< f, \psi_{k,n}  \right> \psi_{k,n}.
\end{align*}
To verify that the frame is still tight, write
\begin{align*}
 \norm{[T_{\mathcal{\overline{M}}}(f)]_n}_2^2 &= \sum_k \abs{\left< [T_{\mathcal{\overline{M}}}(f)]_n, \phi_k \right>}^2 \\
 &= \sum_k \abs{\left< f, \psi_{k,n}\right>}^2,
\end{align*}
so that summing over $n$ gives the result.
\end{proof}

Notice that our Definition \ref{def:adm}  of admissibility is fairly general and directly exploitable for implementing wavelets in the Fourier domain.  However, if a feasible spatial domain representation is required, we must impose certain restrictions. A reasonable condition is to assume that the elements of $\mathcal{M}$ are smooth, since it allows us to relate the multiplier transform to a singular integral transform.

\begin{theorem}\cite[Theorem III.6]{stein70}\label{th:sift}
Let $m$ be homogeneous of degree 0 and indefinitely differentiable on $\mathbb{S}^{d-1}$. Then for $1<p<\infty$ the Fourier multiplier transform $T:L_p(\mathbb{R}^d)\rightarrow L_p(\mathbb{R}^d) $ given by

\begin{equation*}
T(f) = \mathcal{F}^{-1}\{m\widehat{f}\}
\end{equation*}
can be computed by the singular integral

\begin{equation*}
T(f)(\bm{x}) = cf(\bm{x})+\lim_{\epsilon\rightarrow 0}\int_{\abs{\bm{y}}>\epsilon}\frac{\Omega(\bm{y})}{\abs{\bm{y}}^d}f(\bm{x}-\bm{y}){\rm d}\bm{y},
\end{equation*}
where $c$ is the mean value of $m$ on $\mathbb{S}^{d-1}\subset\mathbb{R}^d$ and the functions $m$ and $\Omega$ are related by

\begin{equation*}
m(\bm{\omega})=c + \int_{\mathbb{S}^{d-1}} \left(\frac{\pi i}{2} \text{\rm sign}(\bm{\omega}\cdot\bm{y} )+\log \left(\frac{1}{\abs{\bm{\omega}\cdot\bm{y}}}\right) \right) \Omega(\bm{y}){\rm d} \sigma(\bm{y}).
\end{equation*}
\end{theorem}

\section{Spherical harmonics and singular integral operators}\label{sec:sph_harm}

The class of smooth admissible functions with which we shall be primarily concerned are the spherical harmonics. These can be viewed as a multi-dimensional extension of the trigonometric polynomials in one dimension. The key property is that it is possible to construct an orthonormal basis of $L_2(\mathbb{S}^{d-1})$ using spherical harmonics.  As Fourier multipliers, the spherical harmonics of degree one correspond to the Riesz transform, which has previously been used to construct steerable wavelets \cite{unser11}. It turns out that transforms based on any of the spherical harmonics have similar properties.  One remarkable aspect is the symmetric role assumed by the spatial and frequency variables and the fact that the Fourier transforms are part of the same family, cf. \cite[Chapter IV]{stein71},\cite[Chapter III]{stein70}.  For the benefit of the reader, we review the properties of the spherical harmonics that are relevant for our purpose.

To begin, we recall that a homogeneous polynomial of degree $\ell$ on $\mathbb{R}^d$ is a linear combination of monomials

\begin{equation*}
\sum_{\abs{\bm{\alpha}}=\ell} c_{\bm{\alpha}} \bm{\omega}^{\bm{\alpha}},
\end{equation*}
where the degree $\abs{\bm{\alpha}}=\alpha_1+\cdots+\alpha_d$ of each monomial $\bm{\omega}^{\bm{\alpha}}=\omega_1^{\alpha_1}\cdots \omega_d^{\alpha_d}$ is $\ell$ and each $c_{\bm{\alpha}}$ is complex.  A homogeneous polynomial $P$ is harmonic if it satisfies Laplace's equation

\begin{equation*}
0=\Delta P = \sum_{k=1}^d  \left(\frac{{\rm d} P}{{\rm d}\omega_k} \right)^2.
\end{equation*}
The spherical harmonics of degree $\ell$, denoted by $\mathscr{H}_\ell$, can then be defined as the restrictions to the sphere of such polynomials.  Specifically, every homogeneous harmonic polynomial $P$ of degree ${\ell}$ defines a spherical harmonic $Y\in\mathscr{H}_\ell$ by the equation

\begin{equation*}
Y(\bm{\omega}) = P\left(\bm{\omega}\right)
\end{equation*}
for any $\bm{\omega}\in \mathbb{S}^{d-1}\subset \mathbb{R}^d$.
Moreover, any homogeneous polynomial $P$ of degree $k \geq 0$ can be expanded as

\begin{equation*}
P(\bm{\omega}) = P_0(\bm{\omega}) + \cdots +\abs{\bm{\omega}}^{2\ell}P_{\ell}(\bm{\omega}),
\end{equation*}
where each $P_{j}$ is a homogeneous harmonic polynomial of degree $k-2j$, cf. \cite[Theorem IV.2.1]{stein71}. This means that the proposed steerable wavelet construction includes previous constructions that utilized higher-order Riesz transforms.  

An important property that is implicit in many results is that spherical harmonics of different degrees are orthogonal on the sphere $\mathbb{S}^{d-1}$, cf.   
\cite[Lemma 2]{muller66} or \cite[Corollary IV.2.4]{stein71}.
In particular, since $P_0$ is constant, this implies the spherical harmonics of positive degree have mean zero on the sphere:
\begin{equation*}
\int_{\mathbb{S}^{d-1}} Y(\bm{\omega}) {\rm d} \sigma(\bm{\omega}) = 0 
\end{equation*}
for any $Y\in \mathscr{H}_\ell$ with $\ell>0$.

The dimension of $\mathscr{H}_{\ell}$ can be computed to be 
\begin{equation*}
N(d,\ell) =  \binom{d+\ell-1}{\ell} - \binom{d+\ell-3}{\ell-2}.
\end{equation*}
We shall use the following notation to denote a real-valued orthonormal basis of $\mathscr{H}_{\ell}$:

\begin{equation*}
\left\{ Y_{\ell,k}:\mathbb{S}^{d-1}\rightarrow \mathbb{C}: k =1, \dots, N(d,\ell )   \right\}.
\end{equation*}
Summing over $\ell \leq \ell_{\text{max}}$, we can determine the dimension of the space of spherical harmonics of degree at most $\ell_{\text{max}}$ to be $N(d+1,\ell_{\text{max}})$, cf. 
\cite[p. 4]{muller66} or \cite[Chapter 17]{wendland05}.
Explicit constructions of these basis functions are known. For example, the three-dimensional case is analyzed in detail in \cite[Chapter 9]{beals10}, and a general approach for higher dimensions is given in \cite[Chapter 2]{muller98}.  

We conclude this review with the property that is most important for our purpose: any orthonormal basis $\{Y_{\ell,k}\}$ of $\mathscr{H}_\ell$ satisfies
\begin{equation}\label{eq:repro}
\sum_{k} \frac{\sigma(\mathbb{S}^{d-1})}{N(d,\ell)}\abs{Y_{\ell,k}(\bm{\omega})}^2 = 1,
\end{equation}
cf. 
\cite[Theorem 2]{muller66} or \cite[Corollary IV.2.9]{stein71}.
This is indeed a powerful property, for it implies that the collection of Fourier multipliers

\begin{equation*}
\left\{m_{k}(\bm{\omega}) =\sqrt{\frac{\sigma(\mathbb{S}^{d-1})}{N(d,\ell)}} Y_{\ell,k}\left(\frac{\bm{\omega}}{\abs{\bm{\omega}}}\right) : k=1,\dots, N(d,\ell) \right\}
\end{equation*}
is admissible.  Furthermore, as we shall see shortly, this property allows us to define a collection of admissible multipliers which contains all of the spherical harmonics up to a fixed degree $\ell_{\text{max}}$.

\subsection{Harmonic Riesz transforms }

The motivation for steerable wavelets is to provide a wavelet decomposition which more accurately accounts for local orientation of data.   In a series of papers  \cite{unser13,unser11,unser10}, steerable wavelets have been constructed using the Riesz transform and its higher-order variants, taking advantage of their scale and rotation invariance and their unitary character.  We recall that the Fourier multiplier of the order $\ell$ Riesz transform of an $L_2(\mathbb{R}^d)$ function $f$ is a vector valued function whose components are $(\sqrt{\ell!/\bm{\alpha}!})\bm{\omega}^{\bm{\alpha}}/\abs{\bm{\omega}}^{\ell}$, where  $\abs{\bm{\alpha}}=\ell$.  In the spatial domain, this translates into a principal value singular integral.
The use of spherical harmonics generalizes but also simplifies this construction, thanks to its orthogonality properties.  

To make our construction precise, we designate the Fourier multiplier transforms associated with the spherical harmonics as harmonic Riesz transforms.

\begin{definition}\label{def:hrt}
For any positive integer $\ell_{\text{max}}$ and any unit vector $\bm{c}=(c_0,\dots,c_{\ell_{\text{max}}})\in\mathbb{R}^{\ell_{\text{max}}+1}$, we define the order $\ell_{\text{max}}$ harmonic Riesz transform to be the multiplier transform $T_{\mathcal{M}}$, where

\begin{equation*}
\mathcal{M}= \left\{m_{\ell,k}(\bm{\omega})= c_\ell\sqrt{\frac{\sigma(\mathbb{S}^{d-1})}{ N(d,\ell)}} Y_{\ell,k}\left(\frac{\bm{\omega}}{\abs{\bm{\omega}}}\right) \right\},
\end{equation*}
where $\ell$ ranges from $0$ to $\ell_{\text{max}}$ and $k$ ranges from $1$ to $N(d,\ell)$.
\end{definition}

Note that if $\bm{c}$ contains entries which are zero, the corresponding multipliers are not included in the transform.  The admissibility of this transform follows immediately from Equation \eqref{eq:repro}.  Furthermore, when applied to smooth functions with vanishing moments, the transform preserves decay as well as vanishing moments. 

\begin{theorem}
Let $\psi$ be a differentiable function (or wavelet) with vanishing moments of order $N\geq 1$ such that $\psi$ and its derivatives satisfy the decay estimates 
\begin{enumerate}
\item  $\abs{\psi(\bm{x})} \leq C(1+\abs{\bm{x}})^{-d-N+\epsilon}$,
\item  $\abs{ D^{\bm{\alpha}} \psi(\bm{x})}\leq C (1+\abs{\bm{x}})^{-d-N-1+\epsilon}, \ \abs{\bm{\alpha}}=1$
\end{enumerate}
for some $C>0$ and $0 \leq \epsilon <1$.  Then for any $\ell \geq 0$ and any $1\leq k \leq N(d,\ell)$, the corresponding component of any harmonic Riesz transform $T_{\mathcal{M}}(\psi)$  has decay similar to $\psi$ and maintains the same number of vanishing moments, i.e.
\begin{equation*}
\abs{[T_{\mathcal{M}}(\psi)]_{\ell,k}(\bm{x})}  \leq C (1+\abs{\bm{x}})^{-d-N+\epsilon '}
\end{equation*}
for some $0 \leq \epsilon ' <1$ and $[T_{\mathcal{M}}(\psi)]_{\ell,k}$ has $N$ vanishing moments.
\end{theorem}
\begin{proof}
This follows from the proofs of \cite[Theorems 3.2 and 3.4]{ward13}.  Those results were stated for the first-order Riesz transform; however, one can verify that they hold for a more general class of principal value singular integral operators.  

The kernels of the singular integral operators that define the Riesz transform are $K(\bm{x}) = x_k/\abs{\bm{x}}^{d+1}$, while the kernels used in the harmonic Riesz transforms have the form $\tilde{K}=P(\bm{x})/\abs{\bm{x}}^{d+\ell}$, where $P$ is a homogeneous harmonic polynomial of degree $\ell$.  The essential properties of Riesz kernels that were used in the proof of those theorems were:
\begin{itemize}
\item $K$ has mean zero on the unit sphere;
\item the kernel $K$ and its derivatives satisfy certain decay conditions;
\item $K$ is smooth away from the origin, so that it can be well approximated locally by polynomials.
\end{itemize}
It can be verified that each of these conditions holds for $\tilde{K}$, and hence the results of \cite{ward13} are applicable as well.
\end{proof}

\subsection{Generalized harmonic Riesz transforms}

While directly using the spherical harmonics in a steerable wavelet frame provides a means of categorizing data, we would like to extend this method to make our wavelets adaptable and possibly easier to implement. The approach we take is to compose the harmonic Riesz transforms with matrices representing isometries.  Using an isometry, the derived collection of Fourier multipliers will again be admissible.

\begin{definition}\label{def:ghrt}
Let $T_{\mathcal{M}}$ be a harmonic Riesz transform consisting of $N$ elements. Additionally, let $\mathbf{U}$ be a complex valued matrix of size $n_{\text{max}} \times N$, which represents an isometry; i.e., $\mathbf{U}^T\mathbf{U}$ is the identity matrix of size $N$.  We define the associated generalized harmonic Riesz transform of $f\in L_2(\mathbb{R}^d)$ to be the vector-valued function $T_{\mathcal{M},\mathbf{U}}(f)$ obtained by applying $\mathbf{U}$ to the harmonic Riesz transform $T_{\mathcal{M}}(f)$ of $f$; i.e., the components of $T_{\mathcal{M},\mathbf{U}}(f)$ are linear combinations of the components of $T_{\mathcal{M}}(f)$. 
\end{definition}

Note that in this definition, we deal with two admissible families.  Therefore, we have denoted the size of the original family as $N$, in order to reserve $n_{\text{max}}$ for the size of the derived family. Also, note that
a consequence of the isometry condition is that $n_{\text{max}}\geq N$.  Additionally, if the initial harmonic Riesz transform $T_{\mathcal{M}}$ is defined by a vector $\bm{c}$ with all non-zero entries, then in the above definition $N=N(d+1,\ell_{\text{max}})$. 

\section{Steerable wavelets}\label{sec:steer_wave}

In previous sections we covered the mathematical tools necessary to transform a wavelet frame into a steerable one. In this section, we complete the construction by introducing an appropriate primal wavelet basis. Our choice is the direct extension of the two dimensional case \cite{unser13}.

\begin{proposition}\label{pr:isomom}
Let $h:[0,\infty)\rightarrow \mathbb{R}$ be a smooth function satisfying:

\begin{itemize}
\item[(1)] $h(\omega)=0$ for $\abs{\omega}>1/2$ 
\item[(2)] ${\displaystyle \sum_{j\in \mathbb{Z}} \abs{h(2^j\omega)}^2=1}$
\item[(3)] ${\displaystyle \left.\frac{{\rm d}^nh(\omega)}{{\rm d}\omega^n}\right|_{\omega=0}=0}$ for $n=0,\dots,N$.
\end{itemize}
Then the isotropic mother wavelet $\psi$ whose $d$-dimensional Fourier transform is given by
\begin{equation*}
\widehat{\psi}(\bm{\omega})=h(\abs{\bm{\omega}})
\end{equation*}
generates a tight wavelet frame of $L_2(\mathbb{R}^d)$ whose basis functions
\begin{equation*}
\psi_{j,\bm{k}}(\bm{x})=\psi_j(\bm{x}-2^j\bm{k})  \text{\hspace*{.5cm} with \hspace*{.5cm} } \psi_{j}(\bm{x})=2^{-jd/2}\psi(2^{-j}\bm{x}) 
\end{equation*}
are isotropic with vanishing moments up to order $N$.  Additionally, any $L_2(\mathbb{R}^d)$ function $f$ can be represented as

\begin{equation*}
f = \sum_{j\in\mathbb{Z}}\sum_{\bm{k}\in \mathbb{Z}^d} \left<f,\psi_{j,\bm{k}}\right> \psi_{j,\bm{k}}.
\end{equation*}
\end{proposition} 
\begin{proof}
This follows from a combination of Parseval's identity for Fourier transforms and Plancherel's identity for Fourier series.
\end{proof}

We now define a steerable wavelet frame to be a generalized harmonic Riesz transform of a primal isotropic frame.  Theorem \ref{th:ghrtpf} guarantees that the frame bounds are preserved and that we maintain the reproduction property:

\begin{equation*}
\forall f\in L_2(\mathbb{R}^d), \hspace*{1cm} f = \sum_{j\in\mathbb{Z}}\sum_{\bm{k}\in \mathbb{Z}^d} \sum_{n=1}^{n_{\text{max}}}\left<f,[T_{\mathcal{M},\mathbf{U}}\psi_{j,\bm{k}}]_n\right> [T_{\mathcal{M},\mathbf{U}}\psi_{j,\bm{k}}]_n
\end{equation*}

As we are applying the harmonic Riesz transforms to isotropic functions, the transform can be reduced to a more manageable form.  Essentially, the following is a simplification of Theorem \ref{th:sift}, which uses radial symmetry to reduce the Fourier transform to a one dimensional integral.  

\begin{theorem}\cite[Theorem IV.3.10]{stein71}
Suppose $d\geq 2$ and $\widehat{\psi}\in L^2(\mathbb{R}^d)\cap L^1(\mathbb{R}^d)$ has the form
\begin{equation*}
\widehat{\psi}(\bm{\omega})=h(\abs{\bm{\omega}})\frac{P(\bm{\omega})}{\abs{\bm{\omega}}^\ell},
\end{equation*}
where $P$ is a homogeneous harmonic polynomial of degree $\ell$, then $\psi$ has the form $\psi(\bm{x})=F(\abs{\bm{x}})P(\bm{x})$ where
\begin{equation*}
F(r) = 2\pi i^{\ell} r^{-(d+2\ell-2)/2}\int_0^\infty h(s) J_{(d+2\ell-2)/2}(2\pi rs)s^{d/2}{\rm d}s
\end{equation*}
and $J_{\nu}$ is the Bessel function of the first kind of order $\nu$. 
\end{theorem}

\section{Directional wavelets using zonal harmonics}\label{sec:zonal_harm}

We use the term steerable to convey the fact that the wavelets we construct are intended to be rotated (or steered) to provide a better analysis of the data. To see how this is accomplished, let us consider a generic steerable wavelet $\psi_{\text{Gen}}$; i.e., an element of the generalized harmonic Riesz transform $T_{\mathcal{M},\mathbf{U}}(\psi)$ where $\psi$ is a primal isotropic wavelet.  Each such function is of the form 

\begin{equation*}
\widehat{\psi}_{\text{Gen}}(\bm{\omega}) = h(\abs{\bm{\omega}}) \sum_{\ell = 0}^{\ell_{\text{max}}} \sum_{k=1}^{N(d,\ell)} u_{\ell,k}Y_{\ell,k}\left(\frac{\bm{\omega}}{\abs{\bm{\omega}}}\right)
\end{equation*}
where the coefficients $u_{\ell,k}$ are related to the rows of $\mathbf{U}$.  Of particular interest are the cases where   

\begin{equation*}
u_{\ell,k} = u_\ell Y_{\ell,k}\left(\frac{\bm{\omega}_0}{\abs{\bm{\omega}_0}}\right)
\end{equation*}
for some $\bm{\omega}_0\in\mathbb{R}^d\backslash \{0\}$.  The resulting spherical function is a zonal function, and it has the form 

\begin{equation*}
\sum_{\ell=0}^{\ell_{\text{max}}} u_\ell \frac{N(d,\ell)}{\sigma(\mathbb{S}^{d-1})}P_{\ell}\left(d;\frac{\bm{\omega}_0}{\abs{\bm{\omega}_0}}\cdot \frac{\bm{\omega}}{\abs{\bm{\omega}}}\right) 
\end{equation*}
where $P_{\ell}(d;\cdot)$ is a generalized Legendre polynomial of degree $\ell$, cf. \cite[Section 1.2]{muller98}. A formula for these polynomials is

\begin{equation*}
P_\ell(d;x) = \ell ! \Gamma\left(\frac{d-1}{2} \right) \sum_{l=0}^{\lfloor \ell/2\rfloor} \left(\frac{-1}{4} \right)^l \frac{(1-x^2)^lx^{\ell-2l}}{l!(\ell-2l)!\Gamma(l+(d-1)/2)}.
\end{equation*}
for $x\in[-1,1]$.

One benefit of using these zonal functions is that we can compute rotations fairly effortlessly using the formula:

\begin{equation*}
P_{\ell}(d;\mathbf{R} \bm{\omega}_0 \cdot \mathbf{R} \bm{\omega}  ) = P_{\ell}(d;\bm{\omega}_0 \cdot \bm{\omega}  ),
\end{equation*}
for any matrix $\mathbf{R}$ satisfying $\mathbf{R}^{-1}=\mathbf{R}^{T}$.  In addition to making rotations  straightforward, this structure implies that the value of $\widehat{\psi}_{\text{Gen}}$ is determined solely by the distance of $\bm{\omega}_0/\abs{\bm{\omega}_0}$ from $\bm{\omega}/\abs{\bm{\omega}}$ on the sphere; i.e., it is a zonal function.

Zonal functions have proved to be particularly valuable for approximation on spheres. For example, positive linear combinations of generalized Legendre polynomials are positive semi-definite functions, which can be used for interpolation \cite{schoenberg42,schreiner97,wendland05,xu92}.  Indeed, the steerable wavelet construction we propose uses polynomials which are positive semi-definite on the sphere.

\subsection{Two dimensions}

We would now like to present some information regarding the implementation of steerable wavelets based on generalized Legendre polynomials, and we shall start with a two-dimensional motivating example.  A basis for the circular harmonics of degree $\ell$ on $\mathbb{S}^{1}$ is 
\begin{equation*}
\{\sin(\ell\omega),\cos(\ell\omega)\}.
\end{equation*}
Considering Definition \ref{def:ghrt}, we shall define an isometry to generate a new partition of unity. Note that we shall neglect certain scaling factors, as they do not impact the underlying principle. 

Given a point $\omega_0 \in \mathbb{S}^1$, we define the kernel $P$ in the span of the degree $\ell$ basis by 
\begin{align*}
P(\omega_0,\omega) &= \sin(\ell \omega_0)\sin(\ell \omega) + \cos(\ell \omega_0)\cos(\ell \omega)\\
 &= \cos(\ell(\omega_0-\omega)).
\end{align*}
For a collection of points $\{\omega_n\}_{n=1}^{n_{\text{max}}}$ the matrix that transforms the degree $\ell$ basis $\{\sin(\ell\omega),\cos(\ell\omega)\}$ into $\{\cos(\ell(\omega-\omega_1)),\dots,\cos(\ell(\omega-\omega_{n_{\text{max}}}))\}$ is 

\begin{equation*}
\mathbf{U}_{\ell} =
 \begin{pmatrix}
   \sin(\ell \omega_1) & \cos(\ell \omega_1) \\
   \sin(\ell \omega_2) & \cos(\ell \omega_2) \\
   \vdots & \vdots \\
   \sin(\ell \omega_{n_{\text{max}}}) & \cos(\ell \omega_{n_{\text{max}}}) \\
 \end{pmatrix}.
\end{equation*}
To ensure that the columns of $\mathbf{U}_{\ell}$ are orthogonal, we require

\begin{align*}
0 &= \sum_{n=1}^{n_{\text{max}}} \sin(\ell \omega_n) \cos(\ell \omega_n)\\
 &= \sum_{n=1}^{n_{\text{max}}} \sin(2\ell \omega_n).
\end{align*}
We could consider choosing the points $\omega_n$ to be roots of $\sin(2\ell\cdot)$; however, this approach would be less useful in higher dimensions.  Instead, we shall use a circular quadrature rule with equal weights.  Let $\{\omega_n\}_{n=1}^{n_{\text{max}}}$ be a set of points for which

\begin{equation}\label{eq:2dquad}
\int_{0}^{2\pi} p(\omega) {\rm d}\omega = \frac{2\pi}{n_{\text{max}}} \sum_{n=1}^{n_{\text{max}}} p(\omega_n)
\end{equation}
for all trigonometric polynomials of degree at most $2\ell$. Then

\begin{equation*}
0=\int_{0}^{2\pi} \sin(2\ell\omega){\rm d}\omega.
\end{equation*}
would imply that the columns of $\mathbf{U}_{\ell}$ are orthogonal, and hence that $\mathbf{U}_{\ell}$ is an isometry (after normalization).

Sets of points satisfying Equation \eqref{eq:2dquad} are well known and are referred to as spherical $t$-designs, where $t$ indicates the maximum degree polynomial for which quadrature holds.  Such sets are known to exist for arbitrarily large $t$ \cite{seymour84}, and the most natural choices consist of equidistributed points \cite{bannai09}.  Incidentally, Simoncelli's two dimensional equiangular steerable wavelet construction can be reinterpreted in terms of such $t$-designs \cite{portilla00}.

\subsection{Higher dimensions}

Fix $\ell>0$ and let $\{Y_{\ell,k}\}_{k=1}^{N(d,\ell)}$ be an orthonormal basis for the spherical harmonics of degree  $\ell$ on $\mathbb{S}^{d-1}$.  As in the two-dimensional case, we shall define an isometry to generate a new partition of unity. Specifically, we select $\mathbf{U}_{\ell}$ to be the matrix satisfying

\begin{equation*}
 [\mathbf{U}_{\ell}]_{n,k} = Y_{\ell,k}(\bm{\omega}_n)
\end{equation*}
for a collection of points $X=\{\bm{\omega}_n\}_{n=1}^{n_{\text{max}}} \subset \mathbb{S}^{d-1} $ to be specified. Delsarte et. al refer to this as the $\ell$th characteristic matrix associated with $X$ \cite[Definition 3.4]{delsarte77}.  Applying $\mathbf{U}_\ell$ to the basis generates a new basis of zonal polynomials

\begin{equation*}
\sum_{k=1}^{N(d,\ell)} Y_{\ell,k}(\bm{\omega}_m) Y_{\ell,k}(\bm{\omega}) = \frac{N(d,\ell)}{\sigma(\mathbb{S}^{d-1})}P_{\ell}(d;\bm{\omega}_m \cdot \bm{\omega}  ).
\end{equation*}

In order to make the columns of $\mathbf{U}_{\ell}$ orthogonal, we need

\begin{equation}\label{eq:orthog_cols}
0 =\sum_{n=1}^{n_{\text{max}}} Y_{\ell,k}(\bm{\omega}_n) Y_{\ell,k'}(\bm{\omega}_{n})\\
\end{equation}
for $k\neq k'$.  As in the two-dimensional case, we choose the points $X$ to form a spherical $2\ell$-design, so that 

\begin{equation}\label{eq:t_design}
\frac{\sigma(\mathbb{S}^{d-1})}{n_{\text{max}}}\sum_{n=1}^{n_{\text{max}}} p(\bm{\omega}_n) = \int_{\mathbb{S}^{d-1}}p(\bm{\omega}){\rm d} \sigma(\bm{\omega}) \\
\end{equation}
for spherical harmonics of degree at most $2\ell$.  In fact, the conditions on $X$ considered in \eqref{eq:orthog_cols} and \eqref{eq:t_design} are almost equivalent, cf. \cite[Remark 5.4]{delsarte77}.  On $\mathbb{S}^2$, examples of $t$-designs are provided by the vertices of platonic solids: the vertices of an icosahedron or a dodecahedron constitute $4$-designs \cite{hardin92}.  More generally, $t$-designs appear as orbits of elements of $\mathbb{S}^{d-1}$ under the action of a finite subgroup of the orthogonal group on the sphere \cite{bannai09}. Some specific examples are given in an online library of $t$-designs \cite{hardinwww}. In particular, this library contains $t$-designs on $\mathbb{S}^2$, where $t$ ranges from $0$ to $21$.

\subsection{Localized kernels}

Concerning the construction of zonal basis functions, one final point to address is localization.  The reason for this is that well localized functions can be used to more accurately represent the orientation of data.  While we would ideally like to use locally supported functions, they cannot be represented as polynomials. Therefore, we shall instead use a normalized polynomial approximation of the identity.  

Let us first recall that on the circle, the delta function can be represented by

\begin{equation*}
\delta_0(\omega) = \frac{1}{2\pi}+\frac{1}{\pi}\sum_{\ell=1}^\infty \cos(\ell \omega),
\end{equation*}
and we could construction an approximation by truncating this series,  producing a Dirichlet kernel.  The problem with such a construction is that the Dirichlet kernel is highly oscillatory.  Therefore, we instead propose to construct approximate identities analogous to the scaling functions of Freeden et al. \cite[Section 11.1.3]{freeden98}.  For this construction, we begin with a compactly supported function $\widehat{a}:[0,\infty)\rightarrow [0,1]$ satisfying:
\begin{enumerate}
\item $\widehat{a}(0)=1$, $\widehat{a}(1)=0$, and $\widehat{a}(\omega)>0$ for $\omega \in (0,1)$;
\item $\widehat{a}$ is monotonically decreasing;
\item $\widehat{a}$ is continuous at $0$ and piecewise continuous on $[0,1]$.
\end{enumerate}
Then an approximation to $\delta_0$ is given by

\begin{equation*}
S_{\ell_{\text{max}}}(\omega) = \frac{1}{2\pi}+\frac{1}{\pi}\sum_{\ell=1}^{\ell_{\text{max}}} \widehat{a}\left(\frac{\ell}{\ell_{\text{max}}+1}\right)\cos(\ell \omega).
\end{equation*}

Now, based on the results of the previous section, we know that an admissible zonal basis is given by

\begin{equation*}
\Lambda(\omega-\omega_n) = c_0\sqrt{\frac{1}{n_{\text{max}}}} + \sum_{\ell=1}^{\ell_{\text{max}}} c_\ell \sqrt{\frac{2}{n_{\text{max}}}} \cos(\ell (\omega-\omega_n))
\end{equation*}
for any unit vector $\bm{c}$ and any circular $2 \ell_{\text{max}}$-design $\{\omega_n\}_{n=1}^{n_{\text{max}}}$. Therefore we should choose $\bm{c}$ to be a normalization of the coefficients of $S_{\ell_{\text{max}}}$ to obtain a well localized basis.  

Several examples of such functions appear in the literature. For instance, in 
\cite[Section 11.4.3]{freeden98} the authors propose the cubic polynomial

\begin{equation}\label{eq:cubic_poly_wave}
\widehat{a}(\omega)= (1-\omega)^2(1+2\omega)
\end{equation}
to produce a function $S_{\ell_{\text{max}}}(\omega)$ with suppressed oscillations.  An alternative choice, based on the localization analysis of \cite{narcowich06ltf} and a construction from \cite{mhaskar00}, is to choose $\widehat{a}$ to be a B-spline centered at $0$. 
B-splines also appear implicitly in Simoncelli's steerable pyramid construction \cite{portilla00}.  The angular part of his wavelets are of the form
\begin{align*}
\cos(\omega)^{\ell_{\text{max}}} &= e^{i\omega\ell_{\text{max}}} \sum_{\ell=0}^{\ell_{\text{max}}} \binom{\ell_{\text{max}}}{\ell} e^{-i\omega2\ell} \\
&=\sum_{\ell=0}^{\ell_{\text{max}}} \binom{\ell_{\text{max}}}{\ell} e^{i\omega(\ell_{\text{max}}-2\ell)},
\end{align*}
and the binomial coefficients are the discrete analog of the B-splines.

In higher dimensions we can apply the same analysis to construct localized kernels. This leads us to approximate the delta function at $\bm{\omega}_0$ by

\begin{equation*}
S_{\ell_{\text{max}}}(\bm{\omega}_0\cdot\bm{\omega}) = \sum_{\ell=0}^{\ell_{\text{max}}} \widehat{a}\left(\frac{\ell}{\ell_{\text{max}}+1}\right) \frac{N(d,\ell)}{\sigma(\mathbb{S}^{d-1})} P_{\ell}(d;\bm{\omega}_0\cdot\bm{\omega}).
\end{equation*}
Additionally, our analysis requires that we use kernels of the form

\begin{equation*}
\Lambda(\bm{\omega}_0\cdot\bm{\omega}) = \sum_{\ell =0}^{\ell_{\text{max}}} c_\ell \sqrt{\frac{N(d,\ell)}{n_{\text{max}}}} P_{\ell}(d;\bm{\omega}_0\cdot\bm{\omega});
\end{equation*}
as before, we can adjust the coefficients of $S_{\ell_{\text{max}}}$ to determine an appropriate $\bm{c}$ in $\Lambda$.  Note that the choice of $\widehat{a}$ and $n_{\text{max}}$ must be balanced to produce a good kernel.  Choosing a smoother $\widehat{a}$ produces a kernel with less oscillation; however, its main lobe will be less localized.  To compensate, we could increase the degree $\ell_{\text{max}}$, but this means that we need to increase $n_{\text{max}}$; i.e., utilize a larger collection of basis functions.  Example zonal kernels $\Lambda$ are plotted in Figure \ref{fig:zonal_kernels}.

\begin{figure}
\begin{center}
\includegraphics[width=2.5in]{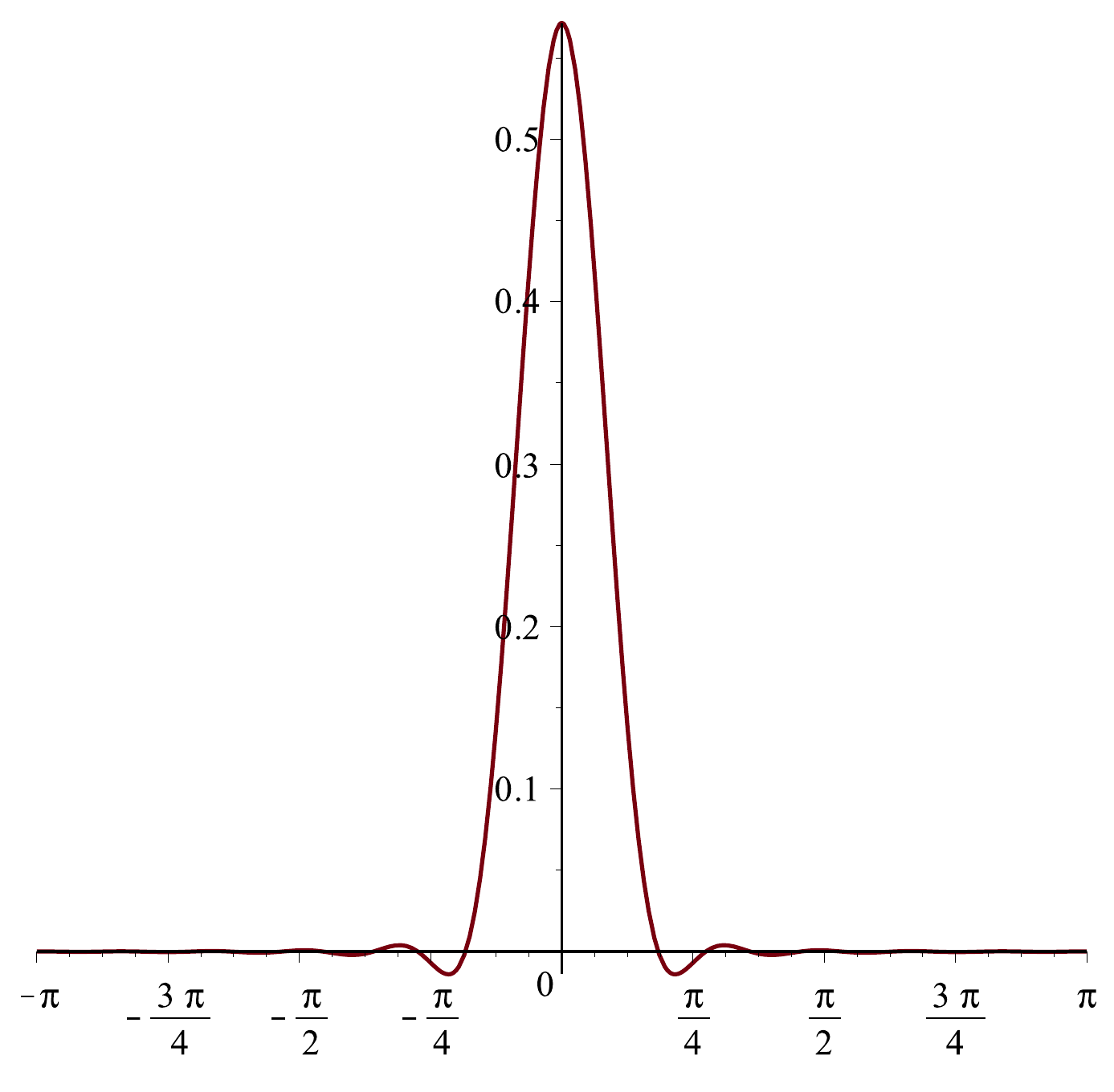}
\hspace{0.25 in}
\includegraphics[width=2.5in]{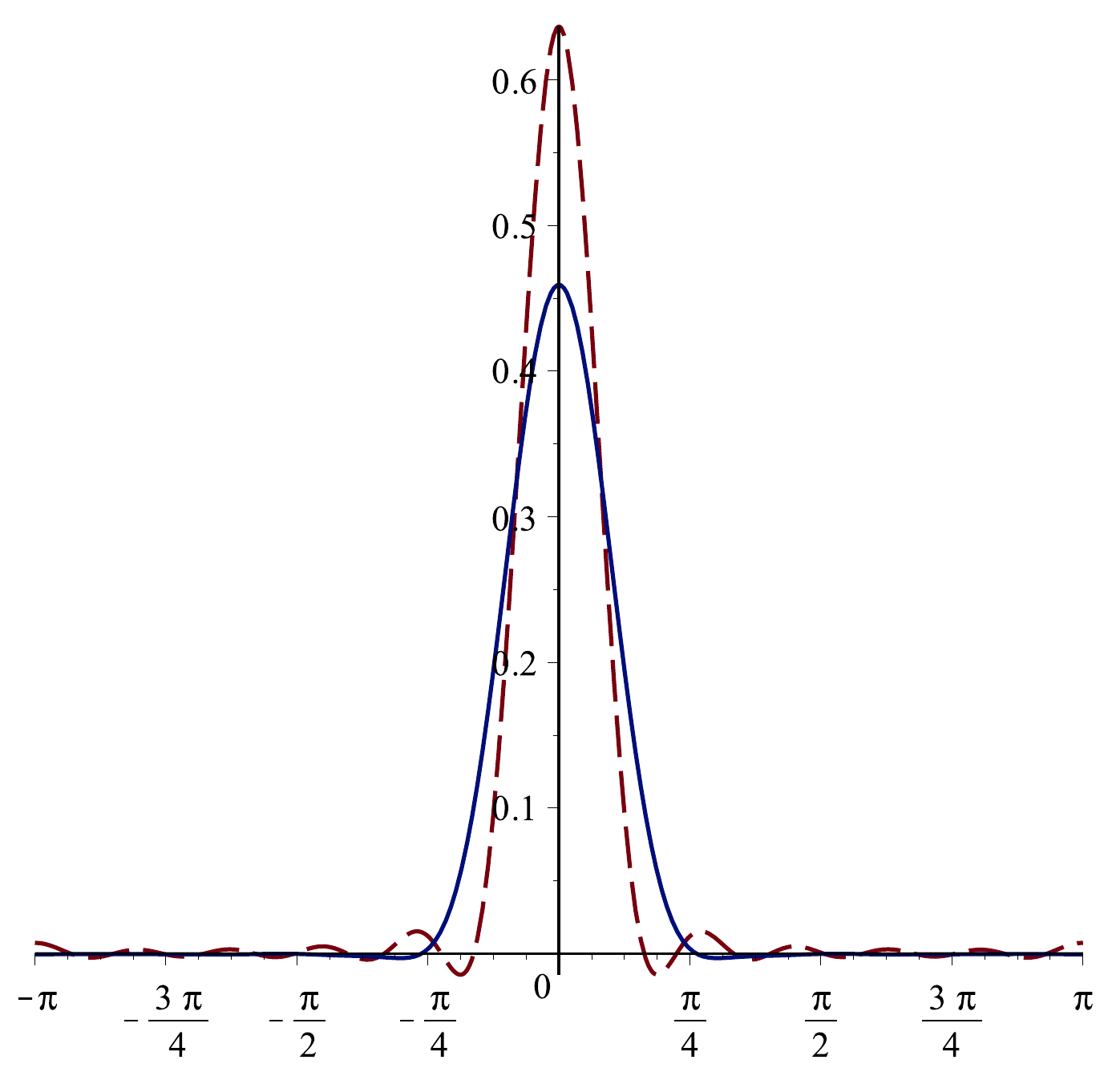}
\end{center}
\caption{Plot of zonal kernels $\Lambda(\cos(\cdot))$ for $d=3$, $\ell_{\text{max}}=10$, and $n_{\text{max}}=216$.  The kernel on the left was constructed using $\widehat{a}$ from \eqref{eq:cubic_poly_wave}, while the functions on the right were constructed by choosing $\widehat{a}$ to be B-splines. In the plot on the right, the dashed line corresponds to a linear B-spline, and the solid line corresponds to a cubic B-spline.}
\label{fig:zonal_kernels}
\end{figure}

If localization is of primary importance, one can define a variance on the sphere, which should be minimized by a polynomial of a given degree. Indeed, this approach was used in \cite{unser13} for steerable wavelets on $\mathbb{R}^2$.  Furthermore, an uncertainty principle for localization in both space and frequency was studied in \cite{freeden98,freeden99ca,narcowich96}.  Related results are contained in \cite{fernandez07}, where the author considers polynomials whose degrees lie within a given range.  Additional work concerning the localization of spherical Slepian functions can be found in \cite{simons10,simons06}.

In the two-dimensional case, Unser et al. extend the definition of variance and allow for a general class of weight functions \cite[Appendix A]{unser13}.  A similar generalized definition of variance was introduced by Michel on $\mathbb{S}^2\subset\mathbb{R}^3$ \cite{michel11}.  However, instead of using Fourier analysis and Bochner's theorem, his results utilize the theory of orthogonal polynomials.  In Appendix A, we provide an extension of this type of approach to the higher dimensional setting that can serve as a basis for the design of wavelets with an optimal angular selectivity.

\section{Construction and steering}

In this section we lay out the construction of the zonal basis and spherical harmonic basis in more detail. Furthermore, we show how the wavelets can be steered using matrix multiplication.  To begin the construction, we choose a maximum degree $\ell_{\text{max}}$ and a real orthonormal basis of spherical harmonics of degree at most $\ell_{\text{max}}$ in a vector $[\mathbf{Y}]_m = Y_m$ of length $N(d+1,\ell_{\text{max}})$.  Then, given a unit vector $\bm{c}=(c_0,\dots,c_{\text{max}})\in\mathbb{R}^{\ell_{\text{max}}+1}$, we define a diagonal matrix $\mathbf{C}$ of size $N(d+1,\ell_{\text{max}}) \times N(d+1,\ell_{\text{max}})$ as follows: For $[\mathbf{Y}]_m$ of degree $\ell$, we set
\begin{equation*}
[\mathbf{C}]_{m,m}= c_\ell \sqrt{\frac{\sigma(\mathbb{S}^{d-1})}{N(d,\ell)}}.
\end{equation*}
This implies that the entries of $\mathbf{C}\mathbf{Y}$ form an admissible collection. Next we choose a collection of points on the sphere
$X_0=\{\bm{\omega}_n\}_{n=1}^{n_{\text{max}}}$, which form a $2\ell_{\text{max}}$-design on $\mathbb{S}^{d-1}$. To construct the final admissible collection, we define the $n_{\text{max}}\times N(d+1,\ell_{\text{max}})$ matrix $\mathbf{U}_{X_0}$ as

\begin{equation*}
[\mathbf{U}_{X_0}]_{n,m} =  \sqrt{\frac{\sigma(\mathbb{S}^{d-1})}{n_{\text{max}}}} Y_m(\bm{\omega}_n).
\end{equation*}
Our admissible collection is then given by $\mathbf{Z}_{X_0}:=\mathbf{U}_{X_0}\mathbf{C}\mathbf{Y}$.

For steering, we use the isometry property of $\mathbf{U}_{X_0}$.  Let $\mathbf{R}$ be a rotation matrix and define $X_1$ to the collection of points obtained by rotating the elements of $X_0$ by $\mathbf{R}$. We then can expand the elements of the rotated basis $\mathbf{Z}_{X_1}$ in terms of the original basis $\mathbf{Z}_{X_0}$ as

\begin{align*}
\mathbf{Z}_{X_1} &= \mathbf{U}_{X_1}\mathbf{C}\mathbf{Y} \\
&= \mathbf{U}_{X_1} \mathbf{U}_{X_0}^T \mathbf{U}_{X_0} \mathbf{C}\mathbf{Y} \\
&= \left(\mathbf{U}_{X_1} \mathbf{U}_{X_0}^T \right) \mathbf{Z}_{X_0}.
\end{align*}

Since the steerable wavelet frame is a Parseval frame, the steering matrix, which transforms the wavelet coefficients from the original basis into the coefficients of the rotated basis, corresponds to the matrix mapping $\mathbf{Z}_{X_0}$ to $\mathbf{Z}_{X_1}$. Hence the $n_{\text{max}}\times n_{\text{max}}$ steering matrix is given by $\mathbf{S}:=\mathbf{U}_{X_1} \mathbf{U}_{X_0}^T$.  Using the properties of zonal harmonics, we can see that the entries of this matrix are pointwise evaluations of a zonal polynomial $\Lambda_{\ell_{\text{max}}}$ (see Figure \ref{fig:lam_max_10_216} for an example).  Precisely

\begin{align*}
\Lambda_{\ell_{\text{max}}}(\bm{\omega}\cdot\bm{\tilde{\omega}}) &= \frac{\sigma(\mathbb{S}^{d-1})}{n_{\text{max}}} \sum_{m=1}^{N(d+1,\ell_{\text{max}})} Y_m(\bm{\omega})Y_m(\bm{\tilde{\omega}}) \\
&=  \frac{\sigma(\mathbb{S}^{d-1})}{n_{\text{max}}} \sum_{\ell=0}^{\ell_{\text{max}}} \frac{N(d,\ell)}{\sigma(\mathbb{S}^{d-1})}P_{\ell}(d;\bm{\omega}\cdot\bm{\tilde{\omega}}) \\
&= \sum_{\ell=0}^{\ell_{\text{max}}} \frac{N(d,\ell)}{n_{\text{max}}}P_{\ell}(d;\bm{\omega}\cdot\bm{\tilde{\omega}})
\end{align*} 
and

\begin{equation*}
[\mathbf{S}]_{n_1,n_2} = \Lambda_{\ell_{\text{max}}}(\mathbf{R}\bm{\omega}_{n_1}\cdot\bm{\omega}_{n_2}).
\end{equation*}

\begin{figure}
\begin{center}
\includegraphics[width=2.5in]{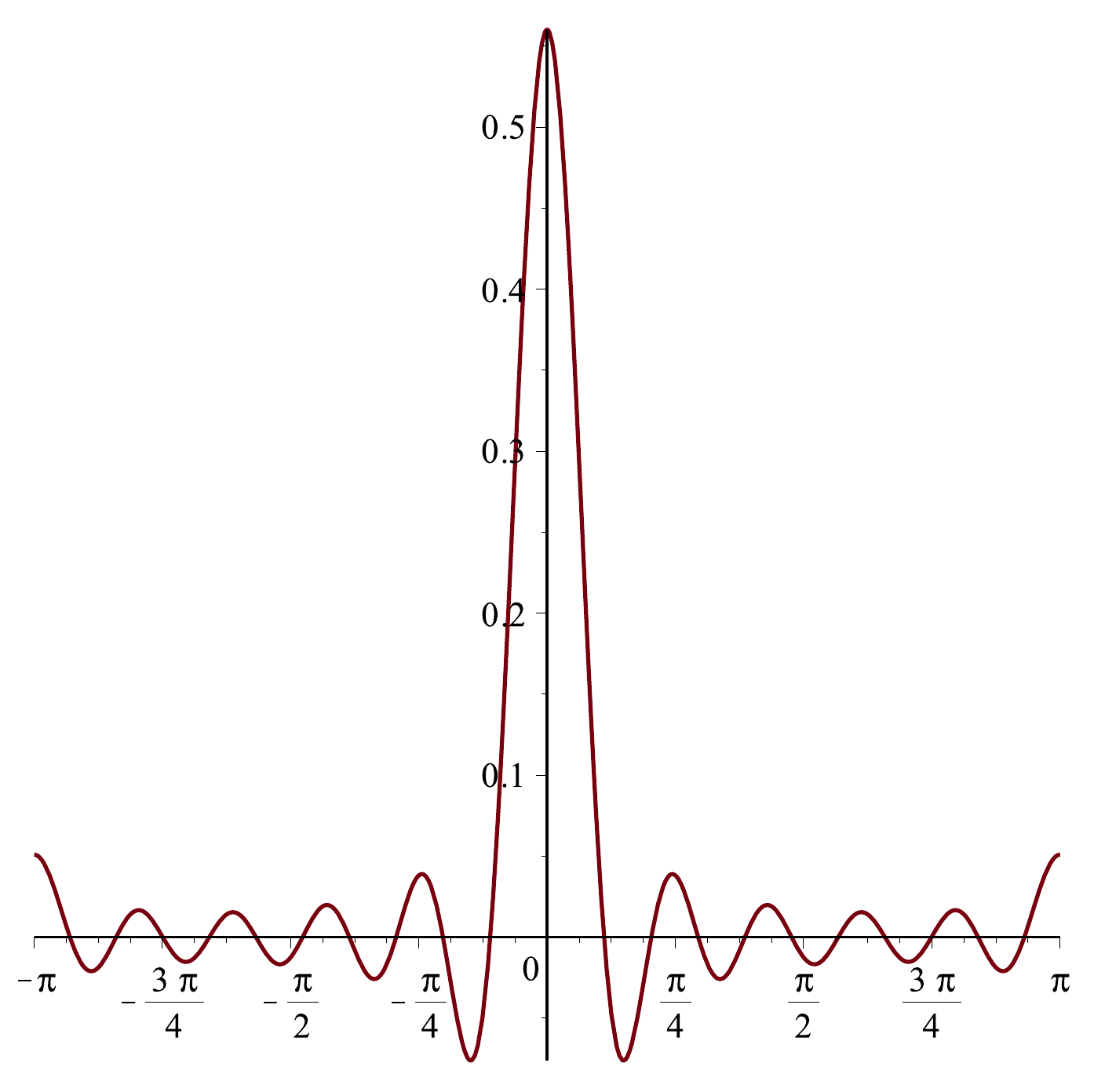}
\end{center}
\caption{Plot of the kernel $\Lambda_{\ell_\text{max}}(\cos(\cdot))$ for $\ell_{\text{max}}=10$ and $n_{\text{max}}=216$.}
\label{fig:lam_max_10_216}
\end{figure}

Interestingly, the steering operation is very much akin to an interpolation that uses $\Lambda_{\ell_{\text{max}}}$ as a kernel (cf. Figure \ref{fig:lam_max_10_216}).  

Depending on the number of points, the steering matrix can be quite large. As an alternative, one could work directly with the orthonormal spherical harmonic basis.  In this situation, the steering matrix reduces significantly as it is a block diagonal matrix with blocks of size $N(d,\ell)$. To see this, let us define the vector $\mathbf{Y}$ by

\begin{equation*}
\mathbf{Y} = \left(
\begin{matrix}
\mathbf{Y}_0      \\
\mathbf{Y}_1      \\
\vdots          \\
\mathbf{Y}_{\ell_{\text{max}}}   
\end{matrix} 
\right),
\end{equation*}
where each $\mathbf{Y}_\ell$ is a vector whose components are an orthonormal basis of the spherical harmonics of degree $\ell$:
\begin{equation*}
\mathbf{Y}_\ell = \left(
\begin{matrix}
Y_{\ell,1}(\bm{\omega})      \\
Y_{\ell,2}(\bm{\omega})       \\
\vdots          \\
Y_{\ell,N(d,\ell)}(\bm{\omega})    
\end{matrix} 
\right).
\end{equation*}
In order to make the entries of $\mathbf{Y}$ an admissible collection, we multiply by the block diagonal matrix

\begin{equation*}
\mathbf{C}=\left(
\begin{matrix}
\mathbf{C}_0 & \mathbf{0}   & \cdots & \mathbf{0}      \\
\mathbf{0}   & \mathbf{C}_1 & \cdots & \mathbf{0}      \\
\vdots       & \vdots       & \ddots & \vdots          \\
\mathbf{0}   & \mathbf{0}   & \cdots & \mathbf{C}_{\ell_{\text{max}}}
\end{matrix} 
\right),
\end{equation*}
where $\mathbf{C}_\ell$ is the $N(d,\ell)\times N(d,\ell)$ diagonal matrix with entries

\begin{equation*}
[\mathbf{C}_\ell]_{k,k}=\sqrt{\frac{\sigma(\mathbb{S}^{d-1})}{(\ell_{\text{max}}+1) N(d,\ell)}}.
\end{equation*}
With $\mathbf{C}$ defined in this way, we are giving equal weight to each degree $\ell$.  One could choose an alternative weighting, but it must be constant over a given degree for the partition of unity property to hold. Now to construct a steering matrix, we need to find an expression for the spherical harmonics in terms of any rotation of them.  Orthogonality between degrees of spherical harmonics makes this especially nice because it means that we can rotate any given degree independently of the others.  Therefore, let us fix $\ell$ and consider the problem of steering the functions in $\mathbf{C}_{\ell}\mathbf{Y}_{\ell}$.  As $\mathbf{C}_\ell$ is a constant multiple of the identity matrix, this is equivalent to finding a steering matrix for $\mathbf{Y}_{\ell}$. Since we are dealing with an orthonormal basis, we can expand any $Y_{\ell,k_0}$ as

\begin{equation*}
Y_{\ell,k_0}(\mathbf{R} \bm{\omega}) = \sum_{k=1}^{N(d,\ell)} \left<Y_{\ell,k_0}(\mathbf{R}\cdot), Y_{\ell,k}   \right> Y_{\ell,k}(\omega),
\end{equation*}
and computation of these inner products can be handled using a quadrature rule.  Specifically, for any $2\ell$-design $\{\omega_n\}_{n=1}^N$, we have

\begin{equation*}
 \left<Y_{\ell,k_0}(\mathbf{R}\cdot), Y_{\ell,k}   \right> = \frac{\sigma(\mathbb{S}^{d-1})}{N}\sum_{n=1}^{N} Y_{\ell,k_0}(\mathbf{R}\bm{\omega}_n) Y_{\ell,k} (\bm{\omega}_n).
\end{equation*}
Therefore, the $N(d,\ell)\times N(d,\ell)$ matrix $\mathbf{V}_\ell$ with entries

\begin{equation*}
[\mathbf{V}_\ell]_{k,k'} = \frac{\sigma(\mathbb{S}^{d-1})}{N}\sum_{n=1}^{N} Y_{\ell,k'}(\mathbf{R}\bm{\omega}_n) Y_{\ell,k} (\bm{\omega}_n)
\end{equation*}
transforms the original basis $\mathbf{Y}_\ell$ into the rotated basis $\mathbf{Y}_\ell^{\mathbf{R}}$, i.e. $\mathbf{V}_\ell \mathbf{Y}_\ell=\mathbf{Y}_\ell^\mathbf{R}$. Consequently, the steering matrix which transforms the wavelet coefficients corresponding to $\mathbf{Y}_\ell$ into the coefficients corresponding to $\mathbf{Y}_\ell^\mathbf{R}$ is also given by  $\mathbf{V}_\ell$.  Note that in two dimensions $N(d,\ell)=2$ for $\ell \geq 1$, and in three dimensions $N(d,\ell)=2\ell+1$. Therefore, these steering matrices can be significantly smaller than the ones used in the zonal construction when a large $\ell_{\text{max}}$ is chosen.

\section{Conclusion and practical summary}

Throughout the course of this paper, we have developed the theory of steerable wavelets in any number of dimensions greater than one.  The previous two sections were devoted to several technical aspects of implementation, and here we summarize the construction. Since the zonal construction is perhaps more tractable conceptually, we now concentrate on this.  One can think of this construction as a generalization for $d>2$ of Simoncelli's equiangular design. 

The substructure of a steerable wavelet frame is an isotropic mother wavelet $\psi$, which satisfies the conditions of Proposition \ref{pr:isomom}.  A new, expanded frame is produced by the collection of mother wavelets 

\begin{equation*}
\left\{ \mathcal{F}^{-1}\{m_n \widehat{\psi} \}: m_n\in \mathcal{M} \right\},
\end{equation*} 
where $\mathcal{M}$ is an admissible class of functions. Now given a maximum degree $\ell_{\text{max}}$, a unit vector $\mathbf{c}=(c_0,\dots,c_{\ell_{\text{max}}} )\in \mathbb{R}^{\ell_{\text{max}}+1}$, and a spherical $2\ell_{\text{max}}$-design $X=\{\bm{\omega}_n\}_{n=1}^{n_{\text{max}}}$, we define the admissible collection

\begin{equation*}
\mathcal{M}=\left\{ m_n(\bm{\omega}) = \sum_{\ell =0}^{\ell_{\text{max}}} c_\ell \sqrt{\frac{N(d,\ell)}{n_{\text{max}}}} P_{\ell}\left(d;\frac{\bm{\omega}_n}{\abs{\bm{\omega}_n}}\cdot\frac{\bm{\omega}}{\abs{\bm{\omega}}}\right) : n=1,\dots,n_{\text{max}}  \right\},
\end{equation*} 
where the $P_{\ell}(d;\cdot)$ are generalized Legendre polynomials. The main properties of this basis are summarized in the following theorem.

\begin{theorem}
The wavelet construction given above defines a tight curvelet-like wavelet frame of $L_2(\mathbb{R}^d)$, where all wavelets are rotated versions of a single wavelet template (per scale). Furthermore, these wavelets are parametrized with a set of coefficients $c_{\ell}$ that can be chosen arbitrarily.
\end{theorem}

Ideally, the vector $\bm{c}$ is chosen so that the functions $m_n$ are well localized and peaked at $\bm{\omega}=\bm{\omega}_n$. For large collections of points $X=\{\bm{\omega}_n\}_{n=1}^{n_{\text{max}}}$ on the sphere and well-localized functions $m_n(\bm{\omega})$, this new frame can detect the orientation of data and provide more information concerning structure. However, one can also work with a smaller collection $\mathcal{M}$ and take advantage of the steering property to orient the wavelet frame in a data-adaptive fashion.  In order to steer the wavelet basis, we need a matrix which transforms the wavelet coefficients upon rotation of the wavelet basis.  The structure of the zonal basis allows us to use the $n_{\text{max}}\times n_{\text{max}}$ steering matrix 

\begin{equation*}
[\mathbf{S}]_{n_1,n_2} = \Lambda_{\ell_{\text{max}}}(\mathbf{R}\bm{\omega}_{n_1}\cdot\bm{\omega}_{n_2})
\end{equation*}
for a rotation $\mathbf{R}$, where

\begin{align*}
\Lambda_{\ell_{\text{max}}}(\bm{\omega}\cdot\bm{\tilde{\omega}}) &=  \sum_{\ell=0}^{\ell_{\text{max}}} \frac{N(d,\ell)}{n_{\text{max}}}P_{\ell}(d;\bm{\omega}\cdot\bm{\tilde{\omega}}).
\end{align*} 
Notice that the entire construction is in terms of the zonal basis, which means we can avoid working directly with the spherical harmonics.  On the other hand, working with the spherical harmonics would allow for smaller steering matrices.

\appendix
\section{Optimal spherical polynomial constructions}

Using the work of Michel \cite{michel11} as a starting point, in this appendix we shall show how orthogonal polynomials can be used to construct spherical polynomials that minimize energy functionals on $\mathbb{S}^{d-1}$. In order to simplify computations let us fix a point $\bm{\omega}_0\in \mathbb{S}^{d-1}$ and assume that each of the generalized Legendre polynomials in $\{P_\ell(d;\bm{\omega}_0\cdot\bm{\omega})\}_{\ell=0}^{\ell_{\text{max}}}$ has been normalized to have $L_2(\mathbb{S}^{d-1})$ norm one.

Considering the framework of our problem, we shall address the problem of finding the polynomial (centered at $\bm{\omega_0}$)

\begin{equation*}
P_{\bm{c}}(\bm{\omega}_0\cdot\bm{\omega}) = \sum_{\ell=0}^{\ell_{\text{max}}} c_\ell P_\ell(d;\bm{\omega}_0\cdot \bm{\omega})
\end{equation*}  
of norm one in $L_2(\mathbb{S}^{d-1})$ that minimizes (or maximizes) an energy functional of the form

\begin{equation*}
E(P_{\bm{c}};W) =  \int_{\mathbb{S}^{d-1}}  \abs{P_{\bm{c}}(\bm{\omega}_0\cdot\bm{\omega})}^2 W(\bm{\omega}_0\cdot\bm{\omega}) {\rm d} \sigma(\bm{\omega}),
\end{equation*}  
where $W$ is an arbitrary positive continuous function. One possibility would be to choose $W(t)=\arccos(t)^2$, so that the energy functional would be

\begin{align*}
E(P_{\bm{c}};\arccos(\cdot)^2) &=  \int_{\mathbb{S}^{d-1}}  \abs{P_{\bm{c}}(\bm{\omega}_0\cdot\bm{\omega})}^2 \arccos(\bm{\omega}_0\cdot\bm{\omega})^2 {\rm d} \sigma(\bm{\omega}) \\
&=\int_{\mathbb{S}^{d-1}}  \abs{P_{\bm{c}}(\bm{\omega}_0\cdot\bm{\omega})}^2 \text{dist}(\bm{\omega}_0,\bm{\omega})^2 {\rm d} \sigma(\bm{\omega}), \\
\end{align*}
where \textit{dist} refers to the spherical distance.  

 As the energy functional contains only zonal functions, it can be reduced to a simpler form:

\begin{align*}
E(P_{\bm{c}};W) &=  \frac{\Gamma(\frac{d}{2})}{\sqrt{\pi}\Gamma(\frac{d-1}{2})} \int_0^{\pi} \abs{P_{\bm{c}}(\cos(\theta))}^2 W(\cos(\theta))\sin^{d-2}(\theta){\rm d} \theta \\
&=  \frac{\Gamma(\frac{d}{2})}{\sqrt{\pi}\Gamma(\frac{d-1}{2})} \int_{-1}^{1} \abs{P_{\bm{c}}(t)}^2 W(t)(1-t^2)^{(d-3)/2}{\rm d} t.
\end{align*}
Now notice that any polynomial of degree at most $\ell_{\text{max}}$ is a linear combination of generalized Legendre polynomials $\{P_\ell(d;\cdot): \ell=0\dots,\ell_{\text{max}}\}$.  Also, the assumptions on $W$ imply that there exists a sequence of polynomials $\{Q_\ell\}_{\ell=0}^{\ell_{\text{max}}}$ that are orthonormal with respect to $E$,

\begin{equation*}
 \frac{\Gamma(\frac{d}{2})}{\sqrt{\pi}\Gamma(\frac{d-1}{2})} \int_{-1}^{1} Q_\ell(t) Q_{\ell'}(t) W(t)(1-t)^{(d-3)/2}{\rm d} t =\delta_{\ell,\ell'},
\end{equation*}
where each $Q_\ell$ is a polynomial of degree $\ell$, cf. \cite{jackson33}. Hence, there exists an invertible change of basis matrix from 
 $\{P_\ell(d;\cdot) \}_{\ell=0}^{\ell_{\text{max}}}$ to  $\{Q_\ell\}_{\ell=0}^{\ell_{\text{max}}}$. Consequently, the polynomial $P_{\bm{c}}$ that minimizes the functional $E(P_{\bm{c}};W)$ is determined by setting $\bm{c}$ to be the unit eigenvector corresponding to the minimal eigenvalue (in absolute value) of the change of basis matrix.

\bibliographystyle{plain}
\bibliography{arxiv_01}

\begin{thebibliography}{10}

\bibitem{bannai09}
Ei. Bannai and Et. Bannai.
\newblock A survey on spherical designs and algebraic combinatorics on spheres.
\newblock {\em Eur. J. Combin.}, 30(6):1392--1425, 2009.

\bibitem{beals10}
R.~Beals and R.~Wong.
\newblock {\em Special functions}, volume 126 of {\em Cambridge Studies in
  Advanced Mathematics}.
\newblock Cambridge University Press, Cambridge, 2010.

\bibitem{calderon57}
A.~P. Calder{{\'o}}n and A.~Zygmund.
\newblock Singular integral operators and differential equations.
\newblock {\em Am. J. Math.}, 79:901--921, 1957.

\bibitem{delsarte77}
P.~Delsarte, J.~M. Goethals, and J.~J. Seidel.
\newblock Spherical codes and designs.
\newblock {\em Geometriae Dedicata}, 6(3):363--388, 1977.

\bibitem{fernandez07}
N.~L. Fern{{\'a}}ndez.
\newblock Optimally space-localized band-limited wavelets on {$\mathbb
  S^{q-1}$}.
\newblock {\em J. Comput. Appl. Math.}, 199(1):68--79, 2007.

\bibitem{freeden98}
W.~Freeden, T.~Gervens, and M.~Schreiner.
\newblock {\em Constructive approximation on the sphere}.
\newblock Numerical Mathematics and Scientific Computation. The Clarendon Press
  Oxford University Press, New York, 1998.

\bibitem{freeden99ca}
W.~Freeden and V.~Michel.
\newblock Constructive approximation and numerical methods in geodetic research
  today -- an attempt at a categorization based on an uncertainty principle.
\newblock {\em J. Geodesy}, 73:452--465, 1999.

\bibitem{hardin92}
R.~H. Hardin and N.~J.~A. Sloane.
\newblock New spherical {$4$}-designs.
\newblock {\em Discrete Math.}, 106/107:255--264, 1992.

\bibitem{hardinwww}
R.~H. Hardin and N.~J.~A. Sloane.
\newblock Spherical designs, http://neilsloane.com/sphdesigns/index.html. Last
  Accessed 28 Dec 2012.

\bibitem{jackson33}
D.~Jackson.
\newblock Series of orthogonal polynomials.
\newblock {\em Ann. of Math. (2)}, 34(3):527--545, 1933.

\bibitem{mhaskar00}
H.~N. Mhaskar, F.~J. Narcowich, J.~Prestin, and J.~D. Ward.
\newblock Polynomial frames on the sphere.
\newblock {\em Adv. Comput. Math.}, 13(4):387--403, 2000.

\bibitem{michel11}
V.~Michel.
\newblock Optimally localized approximate identities on the 2-sphere.
\newblock {\em Numer. Func. Anal. Opt.}, 32(8):877--903, 2011.

\bibitem{mikhlin65}
S.~G. Mikhlin.
\newblock {\em Multidimensional singular integrals and integral equations}.
\newblock Translated from the Russian by W.~J.~A. Whyte. Translation edited by
  I.~N. Sneddon. Pergamon Press, Oxford, 1965.

\bibitem{muller66}
C.~M{{\"u}}ller.
\newblock {\em Spherical harmonics}, volume~17 of {\em Lecture Notes in
  Mathematics}.
\newblock Springer-Verlag, Berlin, 1966.

\bibitem{muller98}
C.~M{{\"u}}ller.
\newblock {\em Analysis of spherical symmetries in {E}uclidean spaces}, volume
  129 of {\em Applied Mathematical Sciences}.
\newblock Springer-Verlag, New York, 1998.

\bibitem{narcowich06ltf}
F.~J. Narcowich, P.~Petrushev, and J.~D. Ward.
\newblock Localized tight frames on spheres.
\newblock {\em SIAM J. Math. Anal.}, 38(2):574--594 (electronic), 2006.

\bibitem{narcowich96}
F.~J. Narcowich and J.~D. Ward.
\newblock Nonstationary wavelets on the {$m$}-sphere for scattered data.
\newblock {\em Appl. Comput. Harmon. A.}, 3(4):324--336, 1996.

\bibitem{portilla00}
J.~Portilla and E.~P. Simoncelli.
\newblock A parametric texture model based on joint statistics of complex
  wavelet coefficients.
\newblock {\em Int. J. Comput. Vision}, 40:49--70, 2000.

\bibitem{schoenberg42}
I.~J. Schoenberg.
\newblock Positive definite functions on spheres.
\newblock {\em Duke Math. J.}, 9:96--108, 1942.

\bibitem{schreiner97}
M.~Schreiner.
\newblock On a new condition for strictly positive definite functions on
  spheres.
\newblock {\em P. Am. Math. Soc.}, 125(2):531--540, 1997.

\bibitem{seymour84}
P.~D. Seymour and T.~Zaslavsky.
\newblock Averaging sets: a generalization of mean values and spherical
  designs.
\newblock {\em Adv. Math.}, 52(3):213--240, 1984.

\bibitem{simons10}
F.~J. Simons.
\newblock Slepian functions and their use in signal estimation and spectral
  analysis.
\newblock In Willi Freeden, M.Zuhair Nashed, and Thomas Sonar, editors, {\em
  Handbook of Geomathematics}, pages 891--923. Springer Berlin Heidelberg,
  2010.

\bibitem{simons06}
F.~J. Simons, F.~A. Dahlen, and M.~A. Wieczorek.
\newblock Spatiospectral concentration on a sphere.
\newblock {\em SIAM Rev.}, 48(3):504--536 (electronic), 2006.

\bibitem{stein70}
E.~M. Stein.
\newblock {\em Singular integrals and differentiability properties of
  functions}, volume~30 of {\em Princeton Mathematical Series}.
\newblock Princeton University Press, Princeton, N.J., 1970.

\bibitem{stein93}
E.~M. Stein.
\newblock {\em Harmonic analysis: real-variable methods, orthogonality, and
  oscillatory integrals}, volume~43 of {\em Princeton Mathematical Series}.
\newblock Princeton University Press, Princeton, NJ, 1993.

\bibitem{stein71}
E.~M. Stein and G.~Weiss.
\newblock {\em Introduction to {F}ourier analysis on {E}uclidean spaces},
  volume~32 of {\em Princeton Mathematical Series}.
\newblock Princeton University Press, Princeton, N.J., 1971.

\bibitem{unser13}
M.~Unser and N.~Chenouard.
\newblock A unifying parametric framework for 2{D} steerable wavelet
  transforms.
\newblock {\em SIAM J. Imaging Sci.}, in press.

\bibitem{unser11}
M.~Unser, N.~Chenouard, and D.~Van De~Ville.
\newblock Steerable pyramids and tight wavelet frames in ${L}_2(\mathbf{R}^d)$.
\newblock {\em IEEE T. Image Process.}, 20(10):2705--2721, 2011.

\bibitem{unser10}
M.~Unser and D.~Van De~Ville.
\newblock Wavelet steerability and the higher-order riesz transform.
\newblock {\em IEEE T. Image Process.}, 19(3):636--652, 2010.

\bibitem{ward13}
J.~P. Ward, K.~N. Chaudhury, and M.~Unser.
\newblock Decay properties of {R}iesz transforms and steerable wavelets.
\newblock arXiv:1301.2525.

\bibitem{wendland05}
H.~Wendland.
\newblock {\em Scattered data approximation}, volume~17 of {\em Cambridge
  Monographs on Applied and Computational Mathematics}.
\newblock Cambridge University Press, Cambridge, 2005.

\bibitem{xu92}
Y.~Xu and E.~W. Cheney.
\newblock Strictly positive definite functions on spheres.
\newblock {\em P. Am. Math. Soc.}, 116(4):977--981, 1992.

\end{thebibliography}

\end{document}